\documentclass[12pt,a4paper]{article}
\usepackage{amssymb}
\usepackage{amsmath}
\usepackage{amsthm}
\usepackage{ot1patch}
\usepackage{indentfirst}

\newtheorem{tw}{Theorem}[section]
\newtheorem{pr}[tw]{Proposition}
\newtheorem{lm}[tw]{Lemma}

\newtheorem{df}[tw]{Definition}
\newtheorem{ex}[tw]{Example}

\DeclareMathOperator{\cha}{char}
\DeclareMathOperator{\jac}{jac}
\DeclareMathOperator{\ijac}{Jac}
\DeclareMathOperator{\dgcd}{dgcd}

\newcommand{\arstd}{\renewcommand{\arraystretch}{1.7}}
\newcommand{\arstz}{\renewcommand{\arraystretch}{1}}

\author{Piotr J\k edrzejewicz\\
\normalsize Faculty of Mathematics and Computer Science\\
\normalsize Nicolaus Copernicus University\\
\normalsize Toru\'{n}, Poland}
\title{A characterization of $p$-bases\\ of rings of constants}
\date{}

\begin{document}

\maketitle

\begin{abstract}
We obtain two equivalent conditions for $m$ polynomials 
in $n$ variables to form a $p$-basis of a ring of constants 
of some polynomial $K$-derivation,
where $K$ is a UFD of characteristic $p>0$.
One of these conditions involves jacobians, 
and the second -- some properties of factors.
In the case of $m=n$ this extends the known theorem of Nousiainen,
and we obtain a new formulation of the jacobian conjecture
in positive characteristic.
\end{abstract}

\begin{table}[b]\footnotesize\hrule\vspace{1mm}
\end{table}

\section{Introduction}

In this paper we give the full characterization
of $p$-bases of kernels of polynomial derivations (Theorem~\ref{t2}).
We refer to the sufficient condition $(i)$
and the necessary condition $(iv)$ from~\cite{jaccond}, Theorem~2.3.
Namely, we show that in the case of the polynomial algebra,
the condition $(i)$ is also necessary,
and we strengthen the condition $(iv)$ to make it also sufficient.
The crucial fact we need to realize this aim
is the positive characteristic version 
of Freudenburg's Lemma for $m$ polynomials,
where the zero characteristic version was obtained
in~\cite{keller}, Theorem~4.1.
Note also that the characterization of one-element $p$-bases
was obtained in~\cite{charoneel}, Theorem~4.2.
Let us sketch more precisely all these connections.

\subsection*{One-element $p$-bases}

Nowicki and Nagata in \cite{NN} obtained 
some interesting results concerning
rings of constants (kernels) of $k$-derivations
of the polynomial algebra $k[x,y]$, where $k$ is a field.
They showed that such a ring is always 
(except zero derivation) of the form:

\smallskip

\noindent
-- $k[f]$ for some $f\in k[x,y]$ if $\cha k=0$,

\smallskip

\noindent
-- $k[x^2,y^2,f]$ for some $f\in k[x,y]$ if $\cha k=2$.

\smallskip

\noindent
In the case of $\cha k=p>2$ they gave an example 
of a nonzero derivation, which ring of constants 
is not of the form $k[x^p,y^p,f]$. 

\medskip

The present author has been studied the rings of constants
of derivations in positive characteristic,
especially the properties of such single generators
(that is, one-element $p$-bases over the respective subring).
Hence, we ask, when $k[x_1^p,\dots,x_n^p,f]$ 
is the ring of constants of a $k$-derivation, 
where $\cha k=p>0$ and
$f\in k[x_1,\dots,x_n]\setminus k[x_1^p,\dots,x_n^p]$.
It was proven in~\cite{p-homogeneous} that
if $f$ is homogeneous modulo $p$ of a nonzero degree,
the following conditions are both necessary and sufficient:
\begin{quote}
$(1)$ \ 
$\gcd\big(\frac{\partial f}{\partial x_1},\dots,
\frac{\partial f}{\partial x_n}\big)=1$,

\smallskip

$(2)$ \ 
{\em $f$ has no square factors and no factors from 
$k[x_1^p,\dots,x_n^p]\setminus k$.}
\end{quote}

\smallskip

In \cite{charoneel} the author discussed sufficient conditions
and necessary conditions on various levels of generality.
In particular, the condition $(1)$ appeared to be sufficient 
for arbitrary polynomial $f$.
The final characterization was obtained in~\cite{one-element}.
It was shown that the above condition $(1)$ is, 
in general, also necessary.
The other equivalent condition has been obtained 
in the following form:
\begin{quote}
$(3)$ \ 
{\em for every $b,c\in k[x_1^p,\dots,x_n^p]$ such that $\gcd(b,c)=1$
the polynomial $bf+c$ has no square factors and no factors from 
$k[x_1^p,\dots,x_n^p]\setminus k$.}
\end{quote}

\subsection*{Many-element $p$-bases}

Nousiainen in~\cite{Nousiainen} 
(see \cite{NiitsumaTRU} or \cite{NiitsumaBCP})
proved that, given polynomials 
$f_1,$ $\dots,$ $f_n\in k[x_1,\dots,x_n]$,
where $k$ is a field of characteristic $p>0$,
then the jacobian condition 
$\det\left[\dfrac{\partial f_i}{\partial x_j}\right]\in k\setminus\{0\}$
holds if and only if $f_1,\dots,f_n$ form a $p$-basis 
of $k[x_1,\dots,x_n]$ over $k[x_1^p,\dots,x_n^p]$,
that is, $k[x_1,\dots,x_n]=k[x_1^p,\dots,x_n^p,f_1,\dots,f_n]$.
Note also that the conditions for existence 
of $p$-bases of ring extensions 
have been recently studied by Ono (\cite{Ono1}, \cite{Ono2}).

\medskip

The jacobian condition is an analog 
of the above condition $(1)$ for $n$ polynomials.
Then it is natural to ask, for arbitrary $m\in\{1,\dots,n\}$,
when $m$ polynomials form a $p$-basis 
of a ring of constants of a $k$-derivation
($k[x_1,\dots,x_n]$ is the only such a ring if $m=n$).

\medskip

Consider the following conditions for polynomials 
$f_1,\dots,f_m\in K[x_1,$ $\dots,$ $x_n]$,
where $K$ is a UFD of characteristic $p>0$, and $m>1$:
\begin{quote}
$(i)$ \
$\gcd\big(\left|\begin{array}{ccc}
\frac{\partial f_1}{\partial x_{j_1}}&
\cdots&\frac{\partial f_1}{\partial x_{j_m}}\\
\vdots&&\vdots\\
\frac{\partial f_m}{\partial x_{j_1}}&
\cdots&\frac{\partial f_m}{\partial x_{j_m}}
\end{array}\right|;\;
1\leqslant j_1<\dots<j_m\leqslant n\big)=1$,

\smallskip

$(ii)$ \
{\em the polynomials $f_1,\dots,f_m$ form a $p$-basis 
of a ring of constants of some $K$-derivation,}

\smallskip

$(iii)$ \
{\em the polynomials $f_1,\dots,f_m$ are $p$-independent and\\
$\gcd\big(\left|\begin{array}{cc}
\frac{\partial f_{i_1}}{\partial x_{j_1}}&
\frac{\partial f_{i_1}}{\partial x_{j_2}}\\
\frac{\partial f_{i_2}}{\partial x_{j_1}}&
\frac{\partial f_{i_2}}{\partial x_{j_2}}
\end{array}\right|;\;
1\leqslant j_1<j_2\leqslant n\big)=1$
for every $i_1<i_2$,}

\smallskip

$(iv)$ \
{\em the polynomials $f_1,\dots,f_m$ are $p$-independent and,
for every $h_1,\dots,h_m\in K[x_1^p,\dots,x_n^p]$,
the polynomials $f_1+h_1,\dots,f_m+h_m$ are pairwise coprime,
have no square factors and no (noninvertible) 
factors from $K[x_1^p,\dots,x_n^p]$.}
\end{quote}
The following implications were obtained
in~\cite{jaccond}, Theorem~2.3:
\arstd
$$\begin{array}{ccc}
(i)&\Rightarrow&(ii)\\
\Downarrow&&\Downarrow\\
(iii)&\Rightarrow&(iv).
\end{array}$$  
\arstz
It was clear that the condition $(iv)$ is, 
in general, not sufficient for $(ii)$,
and that the necessarity of $(i)$ remains an open question.

\medskip

The aim of this paper is to prove that 
the above condition $(i)$ is also necessary for $(ii)$,
and to modify the above condition $(iv)$ 
to make it both necessary and sufficient.

\subsection*{Freudenburg's lemma}

The main preparatory result we need to close 
the chain of implications in Theorem~\ref{t2}
is the positive characteristic version of 
Freudenburg's lemma for $m$ polynomials.
The original version of this lemma was presented 
by Freudenburg in~\cite{Freudenburg} 
for one polynomial in two variables over ${\mathbb C}$.

\begin{lm}[Freudenburg]
Given a polynomial $f\in \mathbb{C}[x,y]$,
suppose $g\in \mathbb{C}[x,y]$ is an irreducible non-constant divisor
of both $\frac{\partial f}{\partial x}$ and $\frac{\partial f}{\partial y}$.
Then there exists $c\in \mathbb{C}$ such that $g$ divides $f+c$.
\end{lm}

This fact was generalized to polynomials in $n$ variables over
an arbitrary algebraically closed field of characteristic zero
by van den Essen, Nowicki and Tyc in~\cite{ENT}, Proposition~2.1.

\begin{pr}[van den Essen, Nowicki, Tyc]
Let $k$ be an algebraically closed field of characteristic zero.
Let $P$ be a prime ideal in $k[x_1,\dots,x_n]$ and $f\in k[x_1,\dots,x_n]$.
If for each $i$ the partial derivative $\frac{\partial f}{\partial x_i}$
belongs to $P$, then there exists $c\in k$ such that $f-c\in P$.
\end{pr}

The present author obtained the following generalization
for one polynomial over an arbitrary field (even a UFD).

\begin{tw}[\cite{charoneel}, Theorem 3.1]
Let $K$ be a UFD, let $P$ be a prime ideal of $K[x_1,\dots,x_n]$.
Consider a polynomial $f\in K[x_1,\dots,x_n]$ such that 
$\frac{\partial f}{\partial x_i}\in P$ for $i=1,\dots,n$.

\medskip

\noindent
{\bf a)}
If $\cha K=0$, 
then there exists an irreducible polynomial $W(T)\in K[T]$
such that $W(f)\in P$.

\medskip

\noindent
{\bf b)}
If $\cha K=p>0$, then there exist $b,c\in K[x_1^p,\dots,x_n^p]$
such that $\gcd(b,c)\sim 1$, $b\not\in P$ and $bf+c\in P$.
\end{tw}

The version for $n$ polynomials in $n$ variables
over a field of characteristic zero has the following form.

\begin{tw}[\cite{keller}, Theorem 4.1]
Let $k$ be a field of characteristic zero,
let $f_1,\dots,f_n\in k[x_1,$ $\dots,x_n]$ be arbitrary polynomials,
and let $g\in k[x_1,\dots,x_n]$ be an irreducible polynomial.
The following conditions are equivalent:

\medskip

\noindent
$(i)$ \
$g$ divides $\det\left[\dfrac{\partial f_i}{\partial x_j}\right]$,

\medskip

\noindent
$(ii)$ \
$g^2$ divides $w(f_1,\dots,f_n)$ for some irreducible polynomial
$w\in k[x_1,\dots,x_n]$.
\end{tw}

It is remarkable that as consequence of the above theorem
we obtain the characterization of polynomial endomorphisms
satisfying the jacobian condition as those mapping
irreducible polynomials to square-free polynomials
(\cite{keller}, Theorem~5.1).

\medskip

In this paper we will obtain Freudenburg's lemma 
for $m$ polynomials in positive characteristic in two forms:
Proposition \ref{p2} and Theorem \ref{t1}.

\subsection*{The main theorem and the connection 
with the jacobian conjecture}

In Theorem \ref{t2} we obtain the final characterization
of $p$-bases of rings of constants of polynomial derivations.
In the case of $n$ polynomials in $n$ variables
it supplements the theorem of Nousiainen
with the condition $(3)$ of Theorem~\ref{t2}.
Hence, following Adjamagbo (\cite{Adjamagbo}, 
see \cite{Essen}, 10.3.16, p.~261),
we can reformulate the jacobian conjecture
in positive characteristic in the form:
\begin{quote}
"If polynomials $f_1,\dots,f_n\in \mathbb{F}_p[x_1,\dots,x_n]$
satisfy the condition $(3)$ of Theorem~\ref{t2} 
(with $K=\mathbb{F}_p$)
and $p$ does not divide the degree of the field extension
$\mathbb{F}_p(f_1,\dots,f_n)\subset \mathbb{F}_p(x_1,\dots,x_n)$,
then $\mathbb{F}_p[f_1,\dots,f_n]=\mathbb{F}_p[x_1,\dots,x_n]$."
\end{quote}
By the theorem of Adjamagbo (\cite{Adjamagbo}, 
see \cite{Essen}, Proposition 10.3.17, p.~261),
if the above property holds for all $n\geqslant 1$
and all primes $p$, then the jacobian conjecture is true.

\section{Rings of constants of derivations and $p$-bases}

Let $A$ be a domain (that is, a commutative ring with unity,
without zero divisors) of characteristic $p>0$.
Let $B$ be a subring of $A$, containing $A^p$,
where $A^p=\{a^p,\,a\in A\}$.
As the main example one may consider the polynomial algebra
$A=K[x_1,\dots,x_n]$ and its subalgebra $B=K[x_1^p,\dots,x_n^p]$,
where $K$ is a domain of characteristic $p>0$.

\medskip

We will use the multi-index notation.
Denote:
$$\Omega_m=\{(\alpha_1,\dots,\alpha_m),
0\leqslant\alpha_1,\dots,\alpha_m<p\}.$$
Given elements $f_1,\dots,f_m\in A$, $m\geqslant 1$,
and $\alpha=(\alpha_1,\dots,\alpha_m)\in\Omega_m$,
we put $f^{\alpha}=f_1^{\alpha_1}\ldots f_m^{\alpha_m}$.
If $\alpha=(0,\dots,0)$, then we put $f^{\alpha}=1$.

\medskip

Recall the definition of a $p$-basis (\cite{Matsumura}, p.~269).

\begin{df}
The elements $f_1,\dots,f_m\in A$ are called:

\medskip

\noindent
{\bf a)}
$p$-independent over $B$, if the elements of the form
$f^{\alpha}$, where $\alpha\in\Omega_m$,
are linearly independent over $B$,

\medskip

\noindent
{\bf b)}
a $p$-basis of $R$ over $B$,
where $R$ is a subring of $A$, containing $B$,
if the elements of the form $f^{\alpha}$,
where $\alpha\in\Omega_m$,
form a basis of $R$ as a $B$-module.
\end{df}

A single element $f\in A$ is $p$-independent
over $B$ if and only if $f\not\in B_0$,
where $B_0$ denotes the field of fractions of $B$,
and in this case the degree of the field extension
$B_0\subset B_0(f)$ equals $p$.
In general, the elements $f_1,\dots,f_m\in A$
are $p$-independent over $B$ if and only if
the degree of the field extension
$B_0\subset B_0(f_1,\dots,f_m)$ equals $p^m$.
The elements $f_1,\dots,f_m\in A$
form a $p$-basis of $R$ over $B$
if and only if they are $p$-independent over $B$
and generate $R$ as a $B$-algebra.
Note also that, if the elements $f_1,\dots,f_m\in A$
form a $p$-basis of $R$ over $B$,
then every element $a\in R$ can be uniquely presented
in the form
$$a=\sum_{\alpha\in\Omega_m}b_{\alpha}f^{\alpha},$$
where $b_{\alpha}\in B$ for $\alpha\in\Omega_m$.

\medskip

Given elements $f_1,\dots,f_m\in A$,
we define the following subring of $A$:
$$C_B(f_1,\dots,f_m)=B_0(f_1,\dots,f_m)\cap A=
B_0[f_1,\dots,f_m]\cap A.$$

\medskip

If $d$ is a derivation of $A$, then its kernel 
is called the ring of constants, and is denoted by $A^d$ 
(see \cite{polder} for a general reference 
on derivations and rings of constants).
Recall (\cite{p-homogeneous}, Theorem~1.1
and \cite{eigenvector}, Theorem~2.5,
see also \cite{noterings}, Theorem~3.1)
that every ring of constants $R$ 
of some $B$-derivation of $A$ satisfies the conditions 
$$B\subset R\hspace{3mm}\mbox{and}\hspace{3mm}R_0\cap A=R.$$
Observe that $C_B(f_1,\dots,f_m)$ is contained
in every ring of constants of a $B$-derivation,
containing the elements $f_1,\dots,f_m$.
If $A$ is finitely generated as a $B$-algebra,
then every subring $R\subset A$ 
satisfying the above conditions
is a ring of constants of some $B$-derivation of $A$.
Hence, under this assumption, $C_B(f_1,\dots,f_m)$
is the smallest (with respect to inclusion)
ring of constants of a $B$-derivation
containing the elements $f_1,\dots,f_m$
(in particular, $B_0\cap A$ is the smallest 
ring of constants of a $B$-derivation).
Moreover, in this case, every ring of constants 
of a $B$-derivation is of the form $C_B(f_1,\dots,f_m)$
for some elements $f_1,\dots,f_m\in A$,
which may be chosen $p$-independent over $B$.
For details, see \cite{eigenvector} and \cite{p-homogeneous}.
Note also that if $A$ is a $K$-algebra,
where $K$ is a domain of characteristic $p>0$,
then every $K$-derivation of $A$ is a $B$-derivation,
where $B=KA^p$ (and then $A^p\subset B$).

\medskip

We are especially interested in $p$-bases of rings of constants.
This condition may be formulated in several equivalent forms.

\begin{lm}
\label{l1}
Given arbitrary elements $f_1,\dots,f_m\in A$,
consider the following conditions:

\medskip

\noindent
$(1)$ \
$f_1,\dots,f_m$ form a $p$-basis (over $B$)
of the ring of constants of some $B$-derivation,

\medskip

\noindent
$(2)$ \
$f_1,\dots,f_m$ are $p$-independent over $B$ and
$B[f_1,\dots,f_m]$ is a ring of constants of some $B$-derivation,

\medskip

\noindent
$(3)$ \
$f_1,\dots,f_m$ are $p$-independent over $B$ and
$C_B(f_1,\dots,f_m)=B[f_1,\dots,f_m]$,

\medskip

\noindent
$(4)$ \
$f_1,\dots,f_m$ form a $p$-basis of $C_B(f_1,\dots,f_m)$ over $B$.

\medskip

\noindent
The following implications hold:
$$(1)\Leftrightarrow (2)\Rightarrow (3)\Leftrightarrow (4).$$

\smallskip

Moreover, if $A$ is finitely generated as a $B$-algebra,
then all the conditions $(1)$ -- $(4)$ are equivalent.
\end{lm}

\begin{proof}
$(1)\Rightarrow (2)$ \
If $f_1,\dots,f_m$ form a $p$-basis of a ring $R$,
then $R=B[f_1,\dots,f_m]$. 

\medskip

\noindent
$(2)\Rightarrow (1)$ \
If $f_1,\dots,f_m$ are $p$-independent over $B$,
then $f_1,\dots,f_m$ form a $p$-basis of $B[f_1,\dots,f_m]$.

\medskip

\noindent
$(2)\Rightarrow (3)$ \
Consider the ring $R=B[f_1,\dots,f_m]$. 
We have $R\subset C_B(f_1,\dots,f_m)$.
On the other hand, $C_B(f_1,\dots,f_m)$ is contained 
in every ring of constants of a $B$-derivation,
containing the elements $f_1,\dots,f_m$,
so if $R$ is a ring of constants of some $B$-derivation,
then $R=C_B(f_1,\dots,f_m)$.

\medskip

$(3)\Leftrightarrow (4)$ \
This equivalence follows directly from
the definition of a $p$-basis.

\medskip

If $A$ is finitely generated as a $B$-algebra,
the implication $(3)\Rightarrow (2)$ follows from 
the fact that $C_B(f_1,\dots,f_m)$ 
is a ring of constants of a $B$-derivation.
\end{proof}

The following technical lemma will be useful
in the proof of Lemma~\ref{l6}.
It is a generalization of Lemma~1.3 from~\cite{one-element}.

\begin{lm}
\label{l2}
Assume that $B_0\cap A=B$.
Let $f_1,\dots,f_m\in A$ be $p$-independent over $B$.
Then the following conditions are equivalent:

\medskip

\noindent
$(1)$ \
$C_B(f_1,\dots,f_m)=B[f_1,\dots,f_m]$,

\medskip

\noindent
$(2)$ \
for every $b\in B\setminus \{0\}$
and $(a_{\alpha}\in B,\alpha\in\Omega_m)$,
if $b\mid\sum_{\alpha\in\Omega_m}a_{\alpha}f^{\alpha}$,
then $b\mid a_{\alpha}$ for every $\alpha\in\Omega_m$.
\end{lm}

\begin{proof}
$(1)\Rightarrow (2)$ \
Assume that $C_B(f_1,\dots,f_m)=B[f_1,\dots,f_m]$.
Consider $b\in B\setminus \{0\}$
and $(a_{\alpha}\in B,\alpha\in\Omega_m)$
such that $b\mid\sum_{\alpha\in\Omega_m}a_{\alpha}f^{\alpha}$.

\medskip

The element
$w=\sum_{\alpha\in\Omega_m}\frac{a_{\alpha}}{b}f^{\alpha}$
belongs to $B_0[f_1,\dots,f_m]$ and $A$,
so $w\in C_B(f_1,\dots,f_m)$.
By the assumption, $w\in B[f_1,\dots,f_m]$,
so $w=\sum_{\alpha\in\Omega_m}c_{\alpha}f^{\alpha}$,
where $c_{\alpha}\in B$ for $\alpha\in\Omega_m$.
The elements of the form $f^{\alpha}$, where $\alpha\in\Omega_m$,
are linearly independent over $B$, hence also over $B_0$.
Therefore, the equality
$$\sum_{\alpha\in\Omega_m}\frac{a_{\alpha}}{b}f^{\alpha}=
\sum_{\alpha\in\Omega_m}c_{\alpha}f^{\alpha}$$
yields that, for every $\alpha\in\Omega_m$,
$\frac{a_{\alpha}}{b}=c_{\alpha}$, that is, $b\mid a_{\alpha}$.

\medskip

\noindent
$(2)\Rightarrow (1)$ \
Assume that, for every $b\in B$
and $(a_{\alpha}\in B,\alpha\in\Omega_m)$,
if $b\mid\sum_{\alpha\in\Omega_m}a_{\alpha}f^{\alpha}$,
then $b\mid a_{\alpha}$ for every $\alpha\in\Omega_m$.
Consider arbitrary element $w\in C_B(f_1,\dots,f_m)$,
$w=\sum_{\alpha\in\Omega_m}c_{\alpha}f^{\alpha}$,
where $c_{\alpha}\in B_0$ for $\alpha\in\Omega_m$.
Of course, we can present each $c_{\alpha}$
as $\frac{a_{\alpha}}{b}$,
where $a_{\alpha}\in B$ for $\alpha\in\Omega_m$
and $b$ is a common denominator for all $\alpha\in\Omega_m$,
$b\in B$.
Since $w\in A$, we have
$b\mid\sum_{\alpha\in\Omega_m}a_{\alpha}f^{\alpha}$.
Then, by the assumption, for every $\alpha\in\Omega_m$,
$b\mid a_{\alpha}$, that is, $c_{\alpha}\in A$,
so $c_{\alpha}\in B$ (because $B_0\cap A=B$).
Finally, $w\in B[f_1,\dots,f_m]$.
\end{proof}

The next lemma follows directly from the proof 
of Proposition 3.3 a) in~\cite{eigenvector}.

\begin{lm}
\label{l3}
If the elements $f_1,\dots,f_m\in A$
form a $p$-basis of $C_B(f_1,\dots,f_m)$ over $B$,
then $B_0\cap A=B$.
\end{lm}

\section{A generalization of Freudenburg's lemma}

Throughout this section $K$ is a UFD of characteristic $p>0$
and $A=K[x_1,\dots,x_n]$ is the polynomial $K$-algebra
in $n$ variables. We put $B=K[x_1^p,\dots,x_n^p]$.
We consider arbitrary polynomials $f_1,\dots,f_m\in A$,
where $m\geqslant 1$.
We also denote:
$R_i=B[f_1,\dots,\widehat{f_i},\dots,f_m]$
for $i\in\{1,\dots,m\}$ and
$R_{ij}=B[f_1,\dots,\widehat{f_i},\dots,
\widehat{f_j},\dots,f_m]$
for $i,j\in\{1,\dots,m\}$, $i\neq j$,
where $\widehat{f_i}$ means that
the element $f_i$ is omitted.

\medskip

For arbitrary $j_1,\dots,j_m\in\{1,\dots,n\}$
we denote by $\jac^{f_1,\dots,f_m}_{j_1,\dots,j_m}$
the jacobian determinant of $f_1,\dots,f_m$
with respect to $x_{j_1},\dots,x_{j_m}$.
By $\ijac(f_1,\dots,f_m)$ we denote
the ideal generated by all determinants of the form
$\jac^{f_1,\dots,f_m}_{j_1,\dots,j_m}$,
where $j_1,\dots,j_m\in\{1,\dots,n\}$.
Moreover, following \cite{jaccond},
we introduce the notion of a differential gcd of $f_1,\dots,f_m$:
$$\dgcd(f_1,\dots,f_m)=
\gcd\left(\jac^{f_1,\dots,f_m}_{j_1,\dots,j_m},\;
j_1,\dots,j_m\in\{1,\dots,n\}\right).$$
If $\jac^{f_1,\dots,f_m}_{j_1,\dots,j_m}=0$
for every $j_1,\dots,j_m\in\{1,\dots,n\}$,
then we put $\dgcd(f_1,$ $\dots,$ $f_m)=0$.

\begin{ex}
\label{e1}
Consider arbitrary $j_1,\dots,j_m\in\{1,\dots,n\}$
and $i\in\{1,\dots,m\}$.
Denote by $d_i$ the $K$-derivation of $A$
defined by $$d_i(f)=
\jac^{f_1,\dots,f_{i-1},f,f_{i+1},\dots,f_m}_{j_1,\dots,j_m}$$
for $f\in A$.
Observe that $d_i(f_i)=\jac^{f_1,\dots,f_m}_{j_1,\dots,j_m}$
and $d_i(f_j)=0$ for $j\neq i$.
Hence, $R_i\subset A^{d_i}$.
\end{ex}

The statement a) in the following lemma
is a consequence of the generalized Laplace Expansion Formula.
The statements b) and c) follow directly from a),
compare the proof of implication $(i)\Rightarrow (iii)$
in \cite{jaccond}, Theorem~2.3.

\begin{lm}
\label{l4}
Consider arbitrary pairwise different
$i_1,\dots,i_r\in\{1,\dots,m\}$, where $1\leqslant r\leqslant m$.
Then:

\medskip

\noindent
{\bf a)}
$\ijac(f_1,\dots,f_m)\subset\ijac(f_{i_1},\dots,f_{i_r})$,

\medskip

\noindent
{\bf b)}
$\dgcd(f_{i_1},\dots,f_{i_r})\mid\dgcd(f_1,\dots,f_m)$,
whenever $\dgcd(f_{i_1},\dots,f_{i_r})\neq 0$,

\medskip

\noindent
{\bf c)}
$\dgcd(f_1,\dots,f_m)=0$ if $\dgcd(f_{i_1},\dots,f_{i_r})=0$.
\end{lm}

Let $Q$ be a prime ideal of $A$.
We denote by $\overline{A}$ the factor algebra $A/Q$.
For arbitrary element $a\in A$ we denote by $\overline{a}$
the coset of $a$ in $\overline{A}$, that is, $\overline{a}=a+Q$.
For a subring $T\subset A$ we denote by $\overline{T}$
the canonical homomorphic image of $T$ in $\overline{A}$.

\begin{lm}
\label{l5}
Consider a subring $T\subset A$ such that $B\subset T$.
Given a polynomial $f\in A$,
the element $\overline{f}$ is $p$-dependent over $\overline{T}$
if and only if there exist $b,c\in T$, $b\not\in Q$,
such that $bf+c\in Q$.
\end{lm}

\begin{proof}
The element $\overline{f}$ is $p$-dependent over $\overline{T}$
if and only if it belongs to $(\overline{T})_0$,
that is, it can be presented in the form
$-\frac{\overline{c}}{\overline{b}}$ for some $b,c\in T$
such that $\overline{b}\neq\overline{0}$,
that is, $b\not\in Q$.
We obtain the equality
$\overline{b}\cdot\overline{f}+\overline{c}=\overline{0}$,
hence $bf+c\in Q$.
\end{proof}

The following proposition is 
a positive characteristic analog 
(for arbitrary number of polynomials)
of Lemma~3.1 from~\cite{charoneel}.

\begin{pr}
\label{p1}
{\bf a)}
The inclusion $\ijac(f_1,\dots,f_m)\subset Q$
holds if and only if the following condition
is satisfied for some $i\in\{1,\dots,m\}$:

\medskip

\noindent
$(\ast)$ \
there exist $s_1,\dots,s_m\in A$, where $s_i\not\in Q$,
such that $s_1d(f_1)+\ldots+s_md(f_m)\in Q$
for every $K$-derivation $d$ of $A$.

\medskip

\noindent
{\bf b)}
If the above condition $(\ast)$ is satisfied
for a given $i\in\{1,\dots,m\}$,
then the element $\overline{f_i}$
is $p$-dependent over $\overline{R_i}$.
\end{pr}

\begin{proof}
{\bf a)}
The condition $\ijac(f_1,\dots,f_m)\subset Q$ holds
if and only if the rank of the matrix
\arstd
$$\left[\begin{array}{cccc}
\overline{\,\frac{\partial f_1}{\partial x_1}\,}&
\overline{\,\frac{\partial f_1}{\partial x_2}\,}&\cdots&
\overline{\,\frac{\partial f_1}{\partial x_n}\,}\\
\overline{\,\frac{\partial f_2}{\partial x_1}\,}&
\overline{\,\frac{\partial f_2}{\partial x_2}\,}&\cdots&
\overline{\,\frac{\partial f_2}{\partial x_n}\,}\\
\vdots&\vdots&&\vdots\\
\overline{\,\frac{\partial f_m}{\partial x_1}\,}&
\overline{\,\frac{\partial f_m}{\partial x_2}\,}&\cdots&
\overline{\,\frac{\partial f_m}{\partial x_n}\,}
\end{array}\right]$$
\arstz
over the field $A_0$ is less than $m$.
Repeating the arguments from the proof
of Lemma 3.1 a) in \cite{keller},
we obtain that this is equivalent to $(\ast)$.

\medskip

\noindent
{\bf b)}
Assume that $(\ast)$ holds for a given $i$.
Consider arbitrary $\overline{R_i}$-derivation 
$\delta$ of $\overline{A}$.
Repeating the arguments from the proof
of Lemma 3.1 b) in \cite{keller}, we obtain that $\delta(\overline{f_i})=\overline{0}$.
Hence, $\overline{f_i}$ belongs to the smallest
ring of constants of a $\overline{R_i}$-derivation 
of $\overline{A}$, that is, $\overline{f_i}\in 
(\overline{R_i})_0\cap \overline{A}$, so 
$\overline{f_i}$ is $p$-dependent over $\overline{R_i}$.
\end{proof}

Proposition \ref{p2} and Theorem \ref{t1}
are two forms of the positive characteristic
Freudenburg's lemma for $m$ polynomials.
The first one follows from
Lemma~\ref{l5} and Proposition~\ref{p1}.
It is a generalization of Theorem~3.1 b)
from~\cite{charoneel}.

\begin{pr}
\label{p2}
Consider arbitrary polynomials 
$f_1,\dots,f_m\in K[x_1,\dots,x_n]$,
where $K$ is a UFD of characteristic $p>0$
and $m,n\geqslant 1$.
If $\ijac(f_1,\dots,f_m)\subset Q$,
where $Q$ is a prime ideal of $A$,
then there exist $i\in\{1,\dots,m\}$ and 
$$b,c\in K[x_1^p,\dots,x_n^p,f_1,\dots,\widehat{f_i},\dots,f_m],$$
$b\not\in Q$, such that $bf_i+c\in Q$.
\end{pr}

The following theorem is a generalization
of Proposition~3.3 b) from~\cite{charoneel}
and a positive characteristic analog
of Theorem~4.1 from~\cite{keller}.

\begin{tw}
\label{t1}
Let $A=K[x_1,\dots,x_n]$ be the polynomial $K$-algebra,
where $K$ is a UFD of characteristic $p>0$ and $n\geqslant 1$.
Put $B=K[x_1^p,\dots,x_n^p]$.
Consider arbitrary polynomials $f_1,\dots,f_m\in A$,
where $m\geqslant 1$,
and denote $R_i=B[f_1,\dots,\widehat{f_i},\dots,f_m]$
and $R_{ij}=B[f_1,\dots,\widehat{f_i},\dots,
\widehat{f_j},\dots,f_m]$, $i\neq j$.

\medskip

Then $\dgcd(f_1,\dots,f_m)$
is divisible by an irreducible polynomial $g\in A$
if and only if at least one of
the following conditions hold:

\smallskip

\noindent
$(i)$ \
$g\not\in B$ and $g^2\mid bf_i+c$ for some $i\in\{1,\dots,m\}$
and $b,c\in R_i$ such that $g\nmid b$,

\smallskip

\noindent
$(ii)$ \
$g\in B$ and $g\mid bf_i+c$ for some $i\in\{1,\dots,m\}$
and $b,c\in R_i$ such that $g\nmid b$,

\smallskip

\noindent
$(iii)$ \
$g\mid b_1f_i+c_1$ and $g\mid b_2f_j+c_2$
for some $i,j\in\{1,\dots,m\}$, $i\neq j$,
and $b_1,b_2,c_1,c_2\in R_{ij}$
such that $g\nmid b_1$ and $g\nmid b_2$.
\end{tw}

\begin{proof}
$(\Rightarrow)$ \
We follow the arguments from the proof of implication
$(i)\Rightarrow (ii)$ in \cite{keller}, Theorem~4.1.
Assume that the polynomial $\dgcd(f_1,\dots,f_m)$
is divisible by an irreducible polynomial $g\in A$.
Then $\ijac(f_1,\dots,f_m)\subset (g)$,
where $(g)$ denotes the principal ideal generated by $g$,
so, by Proposition~\ref{p2},
there exist $i\in\{1,\dots,m\}$, $b,c\in R_i$
such that $g\nmid b$ and $g\mid bf_i+c$,
that is, $$bf_i+c=gh\leqno (1)$$ for some $h\in A$.
If $g\in B$, the condition $(ii)$ holds.
If $g\not\in B$ and $g\mid h$, the condition $(i)$ holds.
Now assume that $g\not\in B$ and $g\nmid h$.

\medskip

Consider arbitrary $j_1,\dots,j_m\in\{1,\dots,n\}$
and the derivation $d_i$ from Example~\ref{e1}.
Note that $d_i(b)=0$ and $d_i(c)=0$, because $b,c\in R_i$.
Applying the derivation $d_i$ to both sides of $(1)$
we obtain
$$b\jac^{f_1,\dots,f_m}_{j_1,\dots,j_m}=
gd_i(h)+d_i(g)h.$$
By the assumption, $g\mid\jac^{f_1,\dots,f_m}_{j_1,\dots,j_m}$,
so $g\mid d_i(g)$.

\medskip

Hence, the determinant
$$\jac^{f_1,\dots,f_{i-1},g,f_{i+1},\dots,f_m}_{j_1,\dots,j_m}=d_i(g)$$
is divisible by $g$ for every $j_1,\dots,j_m\in\{1,\dots,n\}$,
so, by Proposition~\ref{p1} a), there exist $s_1,\dots,s_m\in A$,
where $g\nmid s_j$ for some $j\in\{1,\dots,m\}$,
such that $$g\mid s_1d(f_1)+\ldots+s_id(g)+\ldots+s_md(f_m)$$
for every $K$-derivation $d$ of $A$.
Note that the polynomials $s_j$, $j\neq i$,
can not all together be divisible by $g$.
Indeed, in this case we would have $g\nmid s_i$ and $g\mid s_id(g)$,
so $g\mid d(g)$ for every $K$-derivation $d$,
what is not true for $d=\frac{\partial}{\partial x_l}$
such that $\frac{\partial g}{\partial x_l}\neq 0$
(recall that $g\not\in B$).
Thus $g\nmid s_j$ for some $j\neq i$,
so, by Proposition~\ref{p1} b),
$\overline{f_j}$ is $p$-dependent over
$\overline{B}[\overline{f_1},\dots,\overline{f_{i-1}},\overline{g},
\overline{f_{i+1}},\dots,\widehat{\overline{f_j}},$ $\dots,$ 
$\overline{f_m}]=\overline{R_{ij}}$.

\medskip

Since $\overline{f_i}$ is $p$-dependent over
$\overline{R_i}=\overline{R_{ij}[f_j]}$,
we obtain that both $\overline{f_i}$ and $\overline{f_j}$
are $p$-dependent over $\overline{R_{ij}}$.
Then, by Lemma~\ref{l5},
$g\mid b_1f_i+c_1$ and $g\mid b_2f_j+c_2$
for some $b_1,b_2,c_1,c_2\in R_{ij}$
such that $g\nmid b_1$ and $g\nmid b_2$,
and the condition $(iii)$ holds.

\medskip

$(\Leftarrow)$ \
Assume that the condition $(i)$ holds, that is,
$g\not\in B$ and $g^2\mid bf_i+c$ for some $i\in\{1,\dots,m\}$,
$b,c\in R_i$ such that $g\nmid b$.
Then $bf_i+c=g^2h$ for some $h\in A$.
Applying the derivation $d_i$ from Example~\ref{e1}
for arbitrary $j_1,\dots,j_m\in\{1,\dots,n\}$ we obtain that
$b\jac^{f_1,\dots,f_m}_{j_1,\dots,j_m}=2gd_i(g)h+g^2d_i(h)$,
so $g\mid\jac^{f_1,\dots,f_m}_{j_1,\dots,j_m}$.
Hence, $g\mid\dgcd(f_1,\dots,f_m)$.

\medskip

Analogously, if the condition $(ii)$ holds,
then $g\mid\dgcd(f_1,\dots,f_m)$ as well.

\medskip

Finally, assume that the condition $(iii)$ holds,
that is, $g\mid b_1f_i+c_1$ and $g\mid b_2f_j+c_2$
for some $i,j\in\{1,\dots,m\}$, $i\neq j$,
$b_1,b_2,c_1,c_2\in R_{ij}$ such that $g\nmid b_1$ and $g\nmid b_2$.
It is easy to check that if $g\mid h_1$ and $g\mid h_2$,
where $h_1,h_2\in A$, then $g\mid\jac^{h_1,h_2}_{l_1,l_2}$
for every $l_1$, $l_2$, so $g\mid\dgcd(h_1,h_2)$.
Hence,
$$g\mid\dgcd(b_1f_i+c_1,b_2f_j+c_2),$$
and then, by Lemma~\ref{l4} b), c),
$$g\mid\dgcd(f_1,\dots,b_1f_i+c_1,\dots,b_2f_j+c_2,\dots,f_m).$$

\medskip

On the other side, using the arguments from Example~\ref{e1}, 
for arbitrary $j_1,\dots,j_m\in\{1,\dots,n\}$ we have
$$\jac^{f_1,\dots,b_1f_i+c_1,\dots,b_2f_j+c_2,\dots,f_m}_{j_1,\dots,j_m}=
b_1\jac^{f_1,\dots,f_i,\dots,b_2f_j+c_2,\dots,f_m}_{j_1,\dots,j_m}=
b_1b_2\jac^{f_1,\dots,f_i,\dots,f_j,\dots,f_m}_{j_1,\dots,j_m}.$$
Since $g\nmid b_1b_2$, we obtain that 
$g\mid\jac^{f_1,\dots,f_m}_{j_1,\dots,j_m}$,
so, finally, $g\mid\dgcd(f_1,\dots,f_m)$.
\end{proof}

\section{A characterization of $p$-bases of rings of constants}

Recall that if $A$ is a domain, then two elements $a,b\in A$ 
are called associated, and we denote it $a\sim b$,
if $a=bc$ for some invertible element $c\in A$.
An element $a\in A$ is called square-free
if it is not divisible by a square of any noninvertible element.
If $\cha A=p>0$ and $B$ is a subring containing $A^p$,
then an element $a\in A$ is called $B$-free
if it is not divisible by any noninvertible element of $B$.
(Note that if $A^p\subset B$, then an element of $B$ 
is invertible in $A$ if and only if it is invertible in $B$).

\medskip

The following lemma contains some necessary conditions 
for one and two elements of a UFD to form a $p$-basis 
of the respective ring.
The case of one element was obtained in \cite{charoneel}, 
Theorem 4.2, for a polynomial algebra, 
but the arguments remain valid for arbitrary UFD.

\begin{lm}
\label{l6}
Let $A$ be UFD of characteristic $p>0$,
let $B$ be a subring of $A$, containing $A^p$.

\medskip

\noindent
{\bf a)}
If an element $f\in A$ forms a $p$-basis of $C_B(f)$ over $B$,
then $bf+c$ is square-free and $B$-free for every $b,c\in B$
such that $b\neq 0$ and $\gcd(b,c)\sim 1$.

\medskip

\noindent
{\bf b)}
If elements $f_1,f_2\in A$ form a $p$-basis of $C_B(f_1,f_2)$ over $B$,
then $\gcd(b_1f_1+c_1,b_2f_2+c_2)\sim 1$
for every $b_1,b_2,c_1,c_2\in B$ such that $b_i\neq 0$
and $\gcd(b_i,c_i)\sim 1$ for $i=1,2$.
\end{lm}

\begin{proof}
{\bf a)}
Assume that $f$ forms a $p$-basis of $C_B(f)$ over $B$.
Consider arbitrary elements $b,c\in B$
such that $b\neq 0$ and $\gcd(b,c)\sim 1$.
If $h\mid bf+c$ for some noninvertible element $h\in B$,
then, by Lemmas~\ref{l2} and~\ref{l3}, 
$h\mid b$ and $h\mid c$, a contradiction.

\medskip

Now, suppose that $g^2\mid bf+c$ for some 
irreducible element $g\in A$.
Note that $g^p\mid g^{2(p-1)}$ and $g^{2(p-1)}\mid (bf+c)^{p-1}$,
so $g^p\mid (bf+c)^{p-1}$.
We have $(bf+c)^{p-1}=b^{p-1}f^{p-1}+\ldots+c^{p-1}$,
and by Lemmas~\ref{l2} and~\ref{l3} we obtain that
$g^p\mid b^{p-1}$ and $g^p\mid c^{p-1}$,
so $g\mid b$ and $g\mid c$, a contradiction.

\medskip

\noindent
{\bf b)}
Assume that $f_1,f_2\in A$ form a $p$-basis 
of $C_B(f_1,f_2)$ over $B$.
Suppose that $g\mid b_1f_1+c_1$ and $g\mid b_2f_2+c_2$
for some irreducible element $g\in A$ 
and some $b_1,b_2,c_1,c_2\in B$ such that $b_i\neq 0$
and $\gcd(b_i,c_i)\sim 1$ for $i=1,2$.
Then $g^p\mid (b_1f_1+c_1)^{p-1}(b_2f_2+c_2)$, that is,
$$g^p\mid b_1^{p-1}b_2f_1^{p-1}f_2+b_1^{p-1}c_2f_1^{p-1}+\ldots+
c_1^{p-1}b_2f_2+c_1^{p-1}c_2.$$
By Lemmas~\ref{l2} and~\ref{l3} we obtain that 
$g^p\mid b_1^{p-1}b_2$, so $g\mid b_1$ or $g\mid b_2$.
And this is impossible: if $g\mid b_i$, where $i\in\{1,2\}$,
then $g\mid c_i$, because $g\mid b_if_i+c_i$.
\end{proof}

In the following proposition we obtain 
a necessary condition for a $p$-basis,
stronger than $(iv)$ in \cite{jaccond}, Theorem~2.3.
It will follow from Theorem~\ref{t2}, 
that in the case of the polynomial algebra
this condition is also sufficient.

\begin{pr}
\label{p3}
Let $A$ be UFD of characteristic $p>0$,
let $B$ be a subring of $A$, containing $A^p$.
Assume that the elements $f_1,\dots,f_m\in A$ 
form a $p$-basis of $C_B(f_1,\dots,f_m)$ over $B$.
Denote: $R_i=B[f_1,\dots,\widehat{f_i},\dots,f_m]$,
$R_{ij}=B[f_1,\dots,\widehat{f_i},\dots,
\widehat{f_j},\dots,f_m]$.
Then the following conditions hold:

\smallskip

\noindent
$(i)$ \
the element $bf_i+c$ is square-free and $B$-free
for every $i\in\{1,\dots,m\}$ and $b,c\in R_i$ 
such that $b\neq 0$ and $\gcd(b,c)\sim 1$,

\smallskip

\noindent
$(ii)$ \
if $m>1$, then $\gcd(b_1f_i+c_1,b_2f_j+c_2)\sim 1$
for every $i,j\in\{1,\dots,m\}$, $i\neq j$, 
and $b_1,b_2,c_1,c_2\in R_{ij}$ 
such that $b_1,b_2\neq 0$,
$\gcd(b_1,c_1)\sim 1$, $\gcd(b_2,c_2)\sim 1$.
\end{pr}

\begin{proof}
Put $R=C_B(f_1,\dots,f_m)$.
Then $R=B[f_1,\dots,f_m]$ by Lemma~\ref{l1}.
By the definition of a $p$-basis we see that, 
for every $i\in\{1,\dots,m\}$,
the element $f_i$ forms a $p$-basis of $R$ over $R_i$.
On the other hand, $C_{R_i}(f_i)=R$,
so $(i)$ follows from Lemma \ref{l6}~a),
note that $R_i$-free element is $B$-free.
Similarly, for every $i,j\in\{1,\dots,m\}$, $i\neq j$,
the elements $f_i,f_j$ form a $p$-basis of $R$ over $R_{ij}$,
and $C_{R_{ij}}(f_i,f_j)=R$, 
so $(ii)$ follows from Lemma \ref{l6}~b).
\end{proof}

The next lemma will help us apply the conclusions
of Theorem~\ref{t1} in the proof of Theorem~\ref{t2}.

\begin{lm}
\label{l7}
Let $A$ be a UFD of characteristic $p>0$,
let $B$ be a subring of $A$ such that $A^p\subset B$.
Consider arbitrary element $f\in A$.
Assume that $g^{\varepsilon}\mid bf+c$
for some irreducible element $g\in A$,
$\varepsilon\in\{1,2\}$ and $b,c\in B$ such that $g\nmid b$.
Then there exists an element $c'\in B$ such that
$g^{\varepsilon}\mid bf+c'$ and $\gcd(b,c')\sim 1$.
\end{lm}

\begin{proof}
Consider the decompositions of $b$ and $c$ into irreducible factors:
$b=uq_1^{\alpha_1}\dots q_s^{\alpha_s}$,
$c=vr_1^{\beta_1}\dots r_t^{\beta_t}$,
where the elements $u,v\in A$ are invertible, $s,t\geqslant 0$, 
$q_i\not\sim q_j$ for $i\neq j$, $r_i\not\sim r_j$ for $i\neq j$,
$\alpha_i>0$, $\beta_i>0$.
We may assume that there exists $l\geqslant 0$, $l\leqslant s$,
such that $q_i\sim r_i$ if $i\leqslant l$
and $q_i\not\sim r_j$ if $i,j>l$.

\medskip

Put $h=g^pq_{l+1}^p\dots q_s^p$ ($h=g^p$ if $l=s$) and $c'=c+h$.
We have $h\in B$ and $g^{\varepsilon}\mid h$,
so $c'\in B$ and $g^{\varepsilon}\mid bf+c'$.
Note that $q_i\not\sim g$ for each $i$, because $g\nmid b$.
Now, if $i\leqslant l$, then $q_i\mid c$ and $q_i\nmid h$,
so $q_i\nmid c'$.
If $i>l$, $i\leqslant s$, then $q_i\nmid c$ and $q_i\mid h$,
so $q_i\nmid c'$.
Hence $\gcd(b,c')\sim 1$.
\end{proof}

Now we can prove the main result of the paper
-- a characterization of $p$-bases of rings of constants 
of polynomial derivations.

\begin{tw}
\label{t2}
Let $K$ be a UFD of characteristic $p>0$,
consider arbitrary polynomials
$f_1,\dots,f_m\in K[x_1,\dots,x_n]$,
where $m\in\{1,\dots,n\}$.
Denote: $B=K[x_1^p,\dots,x_n^p]$,
$R_i=B[f_1,\dots,\widehat{f_i},\dots,f_m]$,
$R_{ij}=B[f_1,\dots,\widehat{f_i},\dots,
\widehat{f_j},\dots,f_m]$, $i\neq j$.

\medskip

The following conditions are equivalent:

\smallskip

\noindent
$(1)$ \
$\dgcd(f_1,\dots,f_m)\sim 1$,

\medskip

\noindent
$(2)$ \
the polynomials $f_1,\dots,f_m$ form a $p$-basis
of a ring of constants of some $K$-derivation,

\medskip

\noindent
$(3)$ \
the polynomial $bf_i+c$ is square-free and $B$-free
for every $i\in\{1,\dots,m\}$ and $b,c\in R_i$
such that $b\neq0$ and $\gcd(b,c)\sim 1$,
and, if $m>1$, then\\ 
$\gcd(b_1f_i+c_1,b_2f_j+c_2)\sim 1$
for every $i,j\in\{1,\dots,m\}$, $i\neq j$,
and $b_1,b_2,c_1,c_2\in R_{ij}$ 
such that $b_1,b_2\neq 0$, $\gcd(b_1,c_1)\sim 1$ 
and $\gcd(b_2,c_2)\sim 1$.
\end{tw}

\begin{proof}
$(1)\Rightarrow (2)$ \
This implication was established
in \cite{one-element}, Theorem~2.3 for $m=1$
and in \cite{jaccond}, Theorem~2.3 for $m>1$.

\medskip

\noindent
$(2)\Rightarrow (3)$ \
This implication follows from
Lemma~\ref{l1} and Proposition~\ref{p3}.

\medskip

\noindent
$\neg (1)\Rightarrow \neg (3)$ \
Assume that $\dgcd(f_1,\dots,f_m)\not\sim 1$,
so $\dgcd(f_1,\dots,f_m)$ is divisible
by some irreducible polynomial $g\in A$.
Then, by Theorem~\ref{t1}, at least one of
the conditions $(i)$ -- $(iii)$ hold.
In each case, using Lemma~\ref{l7}, we deduce $\neg (3)$.
\end{proof}

\end{document}